\documentclass[10pt, reqno]{amsart}
\allowdisplaybreaks[1]
\usepackage{amsmath}
\usepackage{amssymb}
\usepackage{amsfonts}
\usepackage{verbatim}
\usepackage{amsthm}
\usepackage[usenames]{color}
\usepackage{hyperref}
\usepackage{fullpage}
\makeindex

 \newtheorem{proposition}{Proposition}

\theoremstyle{definition}

\theoremstyle{remark}

\newtheorem{fact*}{Fact}

\newcommand{\Cd}{\mathcal{C}^d}

\newcommand{\Md}{\mathcal{M}^d}
\newcommand{\M}{\mathcal{M}}

\newcommand{\til}{\raise.17ex\hbox{$\scriptstyle\mathtt{\sim}$}}

\newcommand\beq{\begin{equation}}

\newcommand\eeq{\end{equation}}

\newcommand{\bbm}{\left[ \begin{matrix}}
\newcommand{\ebm}{\end{matrix} \right]}
\newcommand{\bpm}{\left( \begin{matrix}}
\newcommand{\epm}{\end{matrix} \right)}
\numberwithin{equation}{section}

\newlength{\Mheight}
\newlength{\cwidth}

\newcommand{\dfn}[1]{{\bf #1}\index{#1}}

\title[A bad free holomorphic function]{An entire free holomorphic function which is unbounded on the row ball}
\author{
J. E. Pascoe 
}
\date{\today}

\subjclass[2010]{47L25, 46L52, 32A70}

\begin{document}


\maketitle

Michael Hartz recently sent the author a text message about a question he had heard from Eli Shamovich on a visit to Waterloo:
Suppose you have a free holomorphic function on the row ball. Is it bounded on every ball of smaller radius?
The question had come up before during a correspondence with John McCarthy, who had been asked the question by Jim Agler in relation to their then
upcoming LMS lecture series, and the author found some cryptic handwritten notes on his phone that gave the answer.
The answer is a resounding ``no" in more than one variable.
We give an entire free holomorphic function which is unbounded on the row ball.


Define the \dfn{$d$-dimensional matrix universe} to be
	$$\Md = \bigcup_{n\in \mathbb{N}} M_n(\mathbb{C})^d.$$
Let $U \subseteq \Md.$
Let $U_n$ denote $U \cap M_n(\mathbb{C})^d.$
Define a \dfn{free function} to be a function $f: U \rightarrow \M$ such that
\begin{enumerate}
	\item $f$ is graded, that is, if $X \in U_n,$ then $f(X) \in \M_n,$
	\item if $\Gamma X_i = Y_i\Gamma$ for all $i=1,\ldots,d,$ then $\Gamma f(X)=f(Y)\Gamma.$ 
\end{enumerate}
Let $\delta$ be a matrix of noncommutative polynomials. We define a \dfn{basic set}
	$$B_\delta= \{X\in \Md |\|\delta(X)\|<1\}.$$
The topology generated by the $B_\delta$ is called the \dfn{free topology} which has been studied elsewhere \cite{globalholo}.
Note that the intersection of two basic sets is again a basic set.
A function which is locally bounded in the free topology is called a \dfn{free holomorphic function}.
We define the \dfn{$H^\infty$} norm on $B_\delta$ of a function $f$ to be
	$$\|f\|_{B_\delta} = \sup_{X\in B_\delta} \|f(X)\|.$$
We define the \dfn{row ball},
	$$\Cd = \{X \in \Md | \|[X_1, \ldots, X_d]\|<1\}.$$
Note that $\Cd$ is a $B_\delta$ for some choice of $\delta.$

\begin{proposition}
Let $d>1.$
There is an entire free function $f: \Md \rightarrow \M$ such that for every $X$ there is a $G_\delta$ containing $X$ such that $f|_{B_\delta}$ is bounded,
but $f$ is not bounded on $\Cd.$

That is, there is an entire free holomorphic function which is not bounded on $\Cd.$
\end{proposition}
\begin{proof}
	Let $p_n$ be a homogeneous polynomial such that $\|p_n\|_{\mathcal{H}^2}=1$
	such that $p_n|_{M_n^d} = 0$ but $p_n \neq 0.$ (Here, the $\mathcal{H}^2$ norm of a polynomial is the square root of
	the sum of the absolute squares of its coefficients. The norm naturally extends
	to formal power series. It was proved classically that $\|f\|_{\mathcal{H}^2}\leq \|f\|_{\Cd}.$ That such polynomials exist in more that one variable is a delightful combinatorial exercise.)
	Now, let 
		$$q_k = \prod_{n\in \mathbb{N}}  p_n^{\lfloor \frac{k}{2^n \textrm{deg } p_n} \rfloor}.$$
	Note that the degree of $q_k$ is less than or equal to $k,$ $\|q_k\|_{\mathcal{H}^2}=1$ and the degrees grow linearly with $k.$
	Finally, construct
		$$f = \sum^\infty_{k=1} q_k.$$
	Note that $f$ is unbounded on $\Cd$ since $\infty= \|f\|_{\mathcal{H}^2}\leq \|f\|_{\Cd}.$ The function $f$ well-defined for all inputs as the terms in the series are eventually zero.
	
	Let $X \in \Md.$
	There are $r, m \geq 1$ such that $X \in r \Cd_m.$
	Let $$B_\delta = r\Cd \cap B_{(2r)^{2^m\textrm{deg }p_m}p_m}.$$
	By construction $X \in B_\delta.$
	Noting that $\|p_n\|_{r\Cd} \leq r^{\deg p_n},$
	\begin{align*}
		\|q_k\|_{B_\delta} 	& =  \|\prod_{n\in \mathbb{N}}  p_n^{\lfloor \frac{k}{2^n \textrm{deg } p_n} \rfloor}\|_{B_\delta} \\
					& \leq  \|\prod_{n\neq m}  p_n^{\lfloor \frac{k}{2^n \textrm{deg } p_n} \rfloor}\|_{r\Cd} \|p_m^{\lfloor \frac{k}{2^m \textrm{deg } p_m}\rfloor}\|_{B_\delta} \\
					& \leq  r^k \|p_m\|_{B_\delta}^{\lfloor \frac{k}{2^m \textrm{deg } p_m} \rfloor} \\
					& \leq r^k (\frac{1}{(2r)^{2^m\textrm{deg }p_m}})^{\frac{k}{2^m \textrm{deg } p_m} - 1} \\
					& \leq r^k (\frac{1}{(2r)^{k-2^m\textrm{deg }p_m}}) \\
					& \leq \frac{r^{2^m\textrm{deg }p_m}}{2^{k}}.
	\end{align*}
	Therefore, $f$ is bounded on $B_\delta$ by  $r^{2^m\textrm{deg }p_m}.$
\end{proof}

Finally, we note the above example is somehow dual to the theorems of Augat-Balasubramanian-McCullough establishing the paucity of compact sets in the free topology \cite{ABM}, and are essentially (the similarity envelope of) a finite collection of points. Note that the constructed function $f$ also points out the limits of the Agler-McCarthy Oka-Weil theorem \cite{globalholo,augat2018}, which states that a free holomorphic function can be uniformly by polynomials on  compact sets in the free topology. In fact, the Agler-McCarthy Oka-Weil theorem therefore essentially says that one may approximate at finitely many points. Note our example $f$ cannot be uniformly approximated by polynomials on $r\overline{\Cd}$ for $r\geq 1.$ That is, $r\overline{\Cd}$ is not compact in the free topology, so their theorem cannot apply. Thus, the same basic flaw in the free topology witnessed in compact sets unfortunately persists in approximation and function theory.

\bibliography{references}
\bibliographystyle{plain}

\end{document}